\documentclass[12pt,reqno]{amsart}
\usepackage[utf8]{inputenc}

\title{New Lower Bounds For Essential Covers Of The Cube}
\author{Igor Araujo \and J\'ozsef Balogh \and Let\'icia Mattos}

\address{Department of Mathematics, University of Illinois at Urbana-Champaign, Urbana, Illinois 61801, USA (I.~Araujo)}\email{igoraa2@illinois.edu}

\address{Department of Mathematics, University of Illinois at Urbana-Champaign, Urbana, Illinois 61801, USA (J.~Balogh)}\email{jobal@illinois.edu}

\address{Freie Universität Berlin and Berlin Mathematical School (BMS/MATH+), Arnimallee 3, 14195 Berlin, Germany (L.~Mattos)}\email{lmattos@zedat.fu-berlin.de}

\thanks{\rule[-.2\baselineskip]{0pt}{\baselineskip}%
	I. Araujo was partially supported by UIUC Campus Research Board RB 22000. \\
	J. Balogh was partially supported by NSF Grant DMS-1764123, Arnold O. Beckman Research Award (UIUC Campus Research Board RB 22000), the Langan Scholar Fund (UIUC), the Simons Fellowship, and NSF RTG Grant DMS-1937241. \\
	L.~Mattos was supported by the Deutsche Forschungsgemeinschaft (DFG, German Research Foundation) under Germany's Excellence Strategy – The Berlin Mathematics Research Center MATH+ (EXC-2046/1, project ID: 390685689).}
\date{}

\usepackage{amsmath}
\usepackage{ mathrsfs }
\usepackage{ dsfont }
\usepackage{indentfirst}
\usepackage{mathtools}
\usepackage{amsfonts}
\usepackage{amsthm}
\usepackage{graphicx}
\usepackage{float}
\usepackage{caption}
\usepackage[english]{babel}
\usepackage{fullpage}
\usepackage{enumerate}
\usepackage{tasks}
\usepackage{color}
\usepackage{theoremref}
\usepackage{relsize}
\usepackage{pgf,tikz} 
\usepackage{mathdots}
\usepackage{setspace}
\usetikzlibrary{arrows} 

\usepackage{ amssymb }
\usepackage{ textcomp }

\newtheorem{theorem}{Theorem}
\newtheorem{lemma}[theorem]{Lemma}
\newtheorem{prop}[theorem]{Proposition}

\newtheorem{conjecture}[theorem]{Conjecture}

\newtheorem{claim}[theorem]{Claim}

\newtheorem{defn}[theorem]{Definition}

\DeclareMathOperator*{\E}{\mathbb{E}}

\def\R{\mathbb{R}}
\def\P{\mathbb{P}}

\def\N{\mathbb{N}}

\def\floor#1{\lfloor #1 \rfloor}

\def\se{\subseteq}
\def\supp{\textup{supp}}

\newcommand{\ip}[2]{\langle #1,#2\rangle}

\newcommand{\eps}{\varepsilon}
\newcommand{\Ex}{\mathbb{E}}
\renewcommand{\Pr}[1]{\mathbb{P}\left(#1\right)}
\newcommand{\Pru}[1]{\mathbb{P}_{x \sim \{0,1\}^n}\left(#1\right)}

\definecolor{lblue}{rgb}{0.5,0.5,1}

\newcommand{\eq}[1]{\begin{equation}\label{eq:#1}}
	\newcommand{\eqe}{\end{equation}}

\newcommand\restr[2]{{
		\left.\kern-\nulldelimiterspace 
		#1 
		\vphantom{\big|} 
		\right|_{#2} 
}}

\begin{document}
	
	\begin{abstract}
		An \emph{essential cover} of the vertices of the $n$-cube $\{0,1\}^n$ by hyperplanes is a mini\-mal covering where no hyperplane is redundant and every variable appears in the equation of at least one hyperplane. Linial and Radhakrishnan gave a construction of an essential cover with $\lceil \frac{n}{2} \rceil + 1$ hyperplanes and showed that $\Omega(\sqrt{n})$ hyperplanes are required. Recently, Yehuda and Yehudayoff improved the lower bound by showing that any essential cover of the $n$-cube contains at least $\Omega(n^{0.52})$ hyperplanes. In this paper, building on the method of Yehuda and Yehudayoff, we prove that $\Omega \left( \frac{n^{5/9}}{(\log n)^{4/9}} \right)$ hyperplanes are needed.
	\end{abstract}
	
	\maketitle
	
	\section{Introduction} \label{sec:intro}
	
	An \emph{essential cover} of the vertices of the $n$-cube $\{0,1\}^n$ by hyperplanes is a minimal covering where every variable appears in the equation of at least one hyperplane and no hyperplane is redundant.
	Linial and Radhakrishnan~\cite{LINIAL2005331} introduced this notion in 2005, which can be formally defined as follows.
	Throughout this paper, $\ip{\cdot}{\cdot}$ denotes the scalar product and $[n]:=\{1,2,\ldots,n\}$.
	A collection of $k$ hyperplanes with equations $\ip{v_i}{x} = \mu_i$, where $x, v_i\in \R^n$, and $\mu_i\in \R$ for $i \in [k]$, forms an \emph{essential cover} of the $n$-cube $\{0,1\}^n$ if 
	\begin{enumerate}
		\item[$(E1)$] For every $x \in \{0, 1\}^n$, $\ip{v_i}{x} = \mu_i$ for some $i \in [k]$;\label{essential-condition:cover}
		\item[$(E2)$] For every $j \in [n]$, there exists $i \in [k]$ such that $v_i$ satisfies $v_{ij} \neq 0$;\label{essential-condition:all-vars}
		\item[$(E3)$] For every $i \in [k]$, there is a vertex $x\in \{0, 1\}^n$ such that $\ip{v_i}{x} = \mu_i$ but $\ip{v_j}{x} \neq \mu_j$ for all $j \neq i$. \label{essential-condition:minimal}
	\end{enumerate}
	
	In the same paper, Linial and Radhakrishnan showed that any essential cover
	has $\Omega(\sqrt{n})$ hyperplanes. However, this bound is far from the best-known constructions, which were also obtained in~\cite{LINIAL2005331}.
	For every $n \in \mathbb{N}$, they constructed essential covers of size $\lceil \frac{n}{2} \rceil +1$. When $n$ is even, one of the covers is given by the hyperplanes defined by the equations $x_1+\ldots+x_n = n/2$ and $x_{2i-1}-x_{2i}=0$, for $i \in [n/2]$.
	This construction can also be adapted for odd $n$.
	Even though no explicit conjecture on the size of the smallest essential cover was made by Linial and Radhakrishnan, $\lceil n/2 \rceil +1$ hyperplanes are likely to be needed. In 2013, Saxton~\cite{SAXTON2013971} formulated the following conjecture concerning permanents of matrices that, if true, would imply that every essential cover has at least $\lceil n/2 \rceil +1$ hyperplanes.
	
	\begin{conjecture}[Saxton~\cite{SAXTON2013971}]
		If $B \in \R^{m \times n}$ is a matrix such that for every $J \subset [m]$ we have $\left| \bigcup_{j \in J} \{i\in[n]: B_{ij}\neq 0 \} \right| \ge 2|J|-1$, then $B$ contains an $m\times m$ submatrix with non-zero permanent.
	\end{conjecture}
	
	The methods employed by Linial and Radhakrishnan~\cite{LINIAL2005331} to show the $\Omega(\sqrt{n})$ bound are purely algebraic. They applied a result of Alon and Füredi~\cite{ALON199379} to show that 
	if the hyperplanes with equations $\ip{v_i}{x} = \mu_i$, for $i \in [k]$, form an essential cover of the $n$-cube then $|\supp(v_i)|<2k$ for every $i \in [k]$. Here, for a vector $v_i \in \mathbb{R}^n$, we denote by $\supp(v_i):= \{ j \in [n]: v_{ij} \neq 0\}$ the \emph{support} of $v_i$, i.e, the indices of the non-zero entries of $v_i$.
	The $\Omega(\sqrt{n})$ bound immediately follows from this result combined with property (E2), which is equivalent to $|\cup_i\supp(v_i)|\ge n$.
	In~\cite{LINIAL2005331}, Linial and Radhakrishnan also applied the Littlewood--Offord Lemma~\cite{littlewood1943number} to obtain an alternative proof for a weaker result that any essential cover has $\Omega(n^{1/3})$ hyperplanes.
	Although applying the Littlewood--Offord Lemma crudely gives a weaker bound, it gives some insight into the role of the support of the vectors in an essential cover.

	The Littlewood--Offord Lemma states that a hyperplane with equation $\ip{v}{x} = \mu$ contains at most $2^{n}/\sqrt{|\supp(v)|}$ points of the $n$-cube $\{0,1\}^n$.
	Thus, when the support is large, the number of vertices covered by the hyperplane is small.
	Let us briefly explain the idea of how this is used to show that every essential cover has $\Omega(n^{1/3})$ hyperplanes.
	Suppose that the hyperplanes with equations $\ip{v_i}{x} = \mu_i$, for $i \in [k]$, form an essential cover of the $n$-cube.
	As every variable appears in the equation of at least one hyperplane, $|\supp(v_i)|$ is at least $n/k$ on average. 
	Thus, on average, every hyperplane covers at most $(k/n)^{1/2}$ proportion of the $n$-cube. 
	If $(k/n)^{1/2}k<1$, then there is an uncovered vertex, and hence we must have $k \ge n^{1/3}$.
	
	Recently, a new approach was introduced by Yehuda and Yehudayoff~\cite{yehuda2021lower} to show that any essential cover of the $n$-cube contains $\Omega(n^{0.52})$ hyperplanes.
	This improves the previous bound of $\Omega(\sqrt{n})$ by Linial and Radhakrishnan~\cite{LINIAL2005331}.
	Their proof applies a lemma of Bang~\cite{bang1951solution}, which was used to solve Tarski’s plank problem in Euclidean spaces.
	Roughly speaking, Bang's lemma states that for a $k\times k$ symmetric matrix $M$ with positive diagonal entries and a vector $\mu \in \mathbb{R}^k$, there exists $\eps \in \{-1,1\}^k$ which is far from the hyperplanes $\ip{M_i}{x}=\mu_i$, for $i \in [k]$.
	The idea is to apply Bang's lemma to $VV^{T}$, where $V$ is a properly normalized matrix, whose rows are associated with the vectors of an essential cover.
	Then, one would hope to show that the vector $V^{T}\eps$ is sufficiently close to some vertex $u \in \{0,1\}^n$, but it is far enough from any of the hyperplanes $\ip{V_i}{x}=\mu_i$.
	This implies that $u$ itself is far from all hyperplanes $\ip{V_i}{x}=\mu_i$, for all $i \in [k]$.
	When making this argument precise, some issues come along the way. For example, one technical difficulty is that we need to control the column norm of $V$ to show that $V^{T}\eps$ is close to some vertex of $\{0,1\}^n$.
	
	The main contribution of this paper is an improvement on the lower bound on the number of hyperplanes needed in an essential cover of the $n$-cube.
	The main idea behind the proof is to combine the proof strategy of Yehuda and Yehudayoff~\cite{yehuda2021lower} with a stronger result regarding covering systems whose associated matrix has small column norm.
	
	\begin{theorem} \label{thm:main}
		An essential cover of the $n$-cube has $\Omega \left( \frac{n^{5/9}}{(\log n)^{4/9}} \right)$ hyperplanes.
	\end{theorem}    
	
	\noindent\textbf{Remark.}~Yehuda and Yehudayoff's proof ~\cite{yehuda2021lower} gives a lower bound of $n^{12/23-o(1)} \approx n^{0.5217-o(1)}$, while Theorem~\ref{thm:main} gives the lower bound of $n^{5/9-o(1)} \ge n^{0.5555}$. \vspace{1mm}
	
	For an essential cover with hyperplane equations $\ip{v_i}{x} = \mu_i$, for $i \in [k]$, denote by $V$ the $k \times n$ matrix whose $i$-th row is given by $v_i$.
	We call $V$ an \emph{essential} matrix and $Vx=\mu$ an \emph{essential covering system}.
	To prove Theorem~\ref{thm:main}, we suppose for contradiction that there exists an essential covering system $Vx=\mu$, where the number of rows in $V$ is $O \left( \frac{n^{5/9}}{(\log n)^{4/9}} \right)$.
	Then, we show that there exists a point $x \in \{0,1\}^n$ which is not covered by any of the equations $\ip{v_i}{x} = \mu_i$.
	The core of the argument is using the following facts.
	\begin{itemize}
		\item Rows with large support cannot cover many vertices (c.f.~Lemma \ref{anti-concentration:sparse-vectors}). 
		\item Rows corresponding to vectors with \emph{many scales} cannot cover many vertices (c.f.~De\-finition~\ref{definition:many_scales} and Lemma~\ref{lem:probBound}). 
		\item If a set of rows have small column norm, then we can find a vertex far from all hyperplanes (c.f.~Lemma~\ref{lem:bang}). Finding this sparse submatrix is possible because the size of the support of each row is bounded by $2k$ (c.f.~Lemma~\ref{lemma:nati}).
	\end{itemize} 
	
	The rest of the paper is structured as follows. In Section~\ref{sec:scales}, we state the Littlewood--Offord anti-concentration inequality and introduce the notion of a vector with many scales (c.f.~Definition~\ref{definition:many_scales}). In the same section, we also state and show a more refined anti-concentration inequality for vectors with many scales (c.f.~Lemma~\ref{lem:probBound}). 
	In Section~\ref{sec:bang}, we show that rows with small column norm cannot cover many vertices (c.f.~Proposition~\ref{prop:bang}). 
	In Section~\ref{sec:structure}, we prove the structural lemma for the matrix $V$ (c.f.~Lemma~\ref{lemma:seconddecomposition}) that will allow us  to explore the core idea mentioned above. Finally, in Section~\ref{sec:proof}, we prove Theorem~\ref{thm:main}. 
	We highlight that the main novelty from the proof in~\cite{yehuda2021lower} is Proposition~\ref{prop:bang}, where we obtain better upper bounds for the probabilities that a randomly selected vertex lies in a hyperplane.
	
	\section{Anti-concentration for vectors}
	\label{sec:scales}
	
	In this section, we introduce two anti-concentration inequalities that are used in the proof of Theorem~\ref{thm:main}.
	The first is the classical Littlewood--Offord Lemma~\cite{littlewood1943number}, which was proved by Erd\H{o}s~\cite{Erds1945OnAL} using Sperner's Theorem.
	The second is an exponential anti-concentration bound of Yehuda--Yehudayoff~\cite{yehuda2021slicing} for vectors with `nearly exponential decay'.
	These are called vectors with \emph{many scales}. 
	
	For $v \in \mathbb{R}^n$, we denote $\supp(v):= \{ i \in [n]: v_i \neq 0\}$ and
	$\P_{x \sim \{0, 1\}^{n}}$ the probability space generated by taking $x \in \{0, 1\}^{n}$ uniformly at random.
	
	\begin{lemma}[Littlewood--Offord] \label{anti-concentration:sparse-vectors} 
		For every $v \in \R^n\setminus \{0\}$  and $a \in \R$, we have
		\[\Pru{\langle x, v \rangle = a} \leq \dfrac{1}{\sqrt{|\supp(v)|}}.\]
	\end{lemma}
	
	The Littlewood--Offord Lemma is tight up to a multiplicative constant. If $v \in \{0,1\}^n$, then the event
	\[\langle x, v \rangle = \Ex (\langle x, v \rangle)\pm 4\sqrt{|\supp(v)|}\]
	occurs with constant probability. By the pigeonhole principle, it follows that there exists some $a \in \mathbb{Z}$ for which $\Pru{\langle x, v \rangle = a} = \Omega(|\supp(v)|^{-\frac{1}{2}})$.
	However, when $v$ has `nearly exponential decay', this bound can be considerably improved.
	For example, when $v=(1,2,\ldots,2^{n-1})$, then $\Pru{\langle x, v \rangle = a} \le 2^{-n}$.
	More generally, when the coordinates of $v$ decay nearly exponentially, then we expect an exponential-type anti-concentration bound. 
	Following \cite{yehuda2021slicing}, we introduce the notion of a vector with \emph{many scales}, which formalizes this notion of nearly exponential decay between the coordinates of a vector.
	
	For $A\se [n]$ and $v \in \mathbb{R}^n$, denote by $v^{A}$ the vector $v$ restricted to the set $A$.
	For $p \ge 1$ and $v \in \mathbb{R}^n$, let $\|v\|_p:= (\sum_{j=1}^n v_j^p)^{1/p}$ be the $\ell_p$-norm of $v$.
	
	\begin{defn} \label{definition:many_scales}
		The vector $v \in \R^n$ has $S$ scales if
		there exists a partition $[n] = P_1\cup \ldots \cup P_S$ for which $v^{(s)}=v^{P_s}$ satisfies
		$$\|v^{(s)}\|_2 \geq C_1 \|v^{(s+1)}\|_2$$
		for every $s < S$,
		where $C_1=4C_0^2$ and $C_0 = 4.706$.
		The smallest scale of $v$ with respect to the partition $P_1\cup \ldots \cup P_S$ is defined to be $v^{(S)}$ and the size of the smallest scale is $\|v^{(S)}\|_2$. 
	\end{defn} 
	
	Note that the definition of having $S$ scales depends on the constant $C_1=4\cdot(4.706)^2$.
	Let us briefly explain the motivation behind this choice.
	For a vector $v \in \mathbb{R}^n$ with $\ell_2$-norm 1, let 
	\[z_r = z_r(x) = \ip{x^{(r)}}{v^{(r)}} - \frac{1}{2}\sum\limits_{i \in P_r} v_i, \]
	where $x^{(r)}=x^{P_r}$.
	Using the Payley--Zygmund inequality (c.f.~Lemma~\ref{payley-zygmund}) and the second moment method, Yehuda and Yehudayoff~\cite{yehuda2021slicing} showed that the event 
	\begin{align}\label{eq::weakconcentration}
		\frac{\|v^{(r)}\|_2}{C_0} \leq |z_r| \leq C_0 \|v^{(r)}\|_2
	\end{align}
	occurs with probability at least $1/C_0$ for each $r \in [S]$, see Claim~\ref{clm:positiveProb}.
	Let $E$ be the set of indices $r$ for which~\eqref{eq::weakconcentration} holds. 
	Using Chernoff's inequality, we can show that with high probability $|E|=\Omega(S)$.
	Moreover, if $v$ is a vector with $S$ scales, then for $r, s \in E$ and $r<s$ we have
	\begin{align*}
		|z_{r}| \geq \frac{\|v^{(r)}\|_2}{C_0}
		\geq 4 C_0 \|v^{(s)}\|_2 \geq 4 |z_{s}|.
	\end{align*}
	Then $C_1$ is chosen to be $4C_0^2$ since it implies that the sequence $(|z_r|)_{r \in E}$ decay exponentially. 
	Once we have this property, we can show that the anti-concentration bound is $\exp(-\Omega(S))$.
	
	Before stating the anti-concentration bound of  Yehuda and Yehudayoff~\cite{yehuda2021slicing} for vectors with $S$ scales, 
	we state the Payley--Zygmund inequality.
	For self-completeness, we include its proof.
	
	\begin{lemma}[Payley--Zygmund inequality]\label{payley-zygmund}
		Let $\theta \in [0,1]$ and $Z$ be a non-negative random variable with finite variance. Then,
		\[\P \big (Z > \theta \E (Z) \big) \ge (1-\theta)^2 \frac{\E(Z)^2}{\E(Z^2)}.\]
	\end{lemma}
	
	\begin{proof}
		The expectation of $Z$ can be written as 
		\[\E(Z) = \E \big (Z\cdot 1_{\{ Z \le \theta \E(Z)\}} \big ) + \E \big (Z\cdot 1_{\{ Z > \theta \E(Z)\}} \big ).\]
		The first term of this sum  is bounded by 
		\[\E\big (Z\cdot 1_{\{ Z \le \theta \E(Z)\}} \big) \le \theta \E(Z).\]
		By the Cauchy--Schwarz inequality, the second term is bounded by
		\[\E\big (Z\cdot 1_{\{ Z > \theta \E(Z)\}}\big) \le \left(\E(Z^2)^{}\right)^{1/2} \cdot \P\big (Z> \theta \E(Z) \big)^{1/2}.\]
		Combining these bounds, we conclude that
		$\E(Z) \le \theta \E(Z) + \left(\E(Z^2)^{}\right)^{1/2} \P\big (Z> \theta \E(Z) \big)^{1/2}$. This implies 
		\[\P \big (Z > \theta \E (Z) \big) \ge (1-\theta)^2 \frac{\E(Z)^2}{\E(Z^2)}. \qedhere\]
	\end{proof}
	
	Let $x \in \{0,1\}^{n}$ be chosen uniformly at random. 
	The next claim states that if $\|v\|_2 = 1$, then $\ip{x}{v}$ is close to $\E(\ip{x}{v})$ with probability bounded away from $0$.
	As we do not have any assumption on $\supp(v)$, observe that we cannot hope to have a probability close to 1.
	
	\begin{claim}[Yehuda--Yehudayoff~\cite{yehuda2021slicing}]
		\label{clm:positiveProb}
		If $v=(v_1, \ldots, v_n) \in \R^n$ has $\ell_2$-norm $\|v\|_2 =1$, then for every $C_0\ge 4.706$ we have that
		$$\Pru{\frac{1}{C_0} \le \Big| \ip{x}{v} - \frac{1}{2}\sum_{i=1}^n v_i \Big| \leq C_0}
		\ge \frac{1}{C_0}.$$
	\end{claim}
	
	\begin{proof}
		Define $Z = \big (\ip{x}{v}- \E(\ip{x}{v}) \big)^2$.
		Let us calculate the expected value of $4Z$.
		Note that 
		\[4Z= \left(\sum_{i=1}^n v_i \left (2x_i-1\right ) \right)^2.\]
		As $\E(2x_i-1)=0$ and $\E\big((2x_i-1)^2\big)=1$, we have
		\[\E (4Z) = \sum_{i=1}^n v_i^2 = 1.\]
		Now let us calculate the second moment of $4Z$. 
		As $\E(2x_i-1) = \E\big((2x_i-1)^3\big)=0$ and $\E\big((2x_i-1)^2\big)=\E\big((2x_i-1)^4\big)=1$, we have
		\[\E \big((4Z)^2\big) = \E \left( \sum_{i=1}^n v_i (2x_i-1) \right)^4 = 
		\sum_{i=1}^n v_i^4 + 6 \sum_{i\neq j} v_i^2v_j^2 \le 3 \left( \sum_{i=1}^n v_i^2 \right)^2 = 3.\]
		
		Let $C \ge 4.706$. As $\E(4Z) = 1$, by Markov's inequality we have
		\[\Pr{4Z \geq 4C^2} \leq \frac{1}{4C^2}.\] 
		As $\E(4Z) = 1$ and $\E \big((4Z)^2\big)\le 3$, by the Payley-Zygmond inequality (c.f.~Lemma~\ref{payley-zygmund}) we have
		\[\P \left (4Z > \dfrac{4}{C^2} \right) \ge \dfrac{1}{3}\left(1-\dfrac{4}{C^2} \right)^2.\]
		Combining the previous inequalities, it follows that
		\[\P \left ( \dfrac{1}{C^2} < Z < C^2 \right) \ge \dfrac{1}{3}\left(1-\dfrac{4}{C^2}\right)^2 -  \frac{1}{4C^2}.\]
		The last expression is at least $C^{-1}$ whenever $C \ge 4.706$.
	\end{proof}
	
	We are now ready to state and prove the anti-concentration bound of Yehuda and Yehudayoff~\cite{yehuda2021slicing} for vectors with $S$ scales.
	Throughout this paper, logarithms are in base $e$.
	
	\begin{lemma}[Yehuda--Yehudayoff~\cite{yehuda2021slicing}]
		\label{lem:probBound}
		There is a constant $C_2 > 1$ such that the following holds.
		If $v \in \R^n$ has $S$ scales and the size of the smallest scale is $\delta > 0$,
		then for every $a \in \R$ and $b \geq 2$ we have
		\[\P_{x \sim \{0, 1\}^n}\big (|\langle x,v \rangle - a| < b \delta\big) < C_2 \exp \left (-\tfrac{S}{C_2} + C_2 \log(b) \right).\]
	\end{lemma}
	
	\begin{proof}
		Let $[n] = P_1\cup \ldots \cup P_S$ be the partition of $[n]$ associated to the scales of $v$. 
		For $r \in [S]$, define
		\[z_r = \ip{x^{(r)}}{v^{(r)}} - \frac{1}{2}\sum\limits_{i \in P_r} v_i, \]
		where $x^{(r)}=x^{P_r}$.
		Let $E$ be the set of indices $r \in [S]$ for which the event
		\begin{align}\label{eq:weakconcentration}
			\frac{\|v^{(r)}\|_2}{C_0} \leq |z_r| \leq C_0 \|v^{(r)}\|_2
		\end{align}
		occurs.
		We claim that  $(|z_r|)_{r \in  E}$ decreases exponentially and that $|z_{r}| > 3b\delta$ whenever $r \in E$ and $r < S-\log(3b)$.
		In fact, note that if $r \in E$ and $s> r$ ($s$ might not be in $E$), then
		\begin{equation} \label{eq::zr>4zs}
			|z_{r}| \geq \frac{\|v^{(r)}\|_2}{C_0}
			\geq 4 C_0 \|v^{(s)}\|_2,
		\end{equation}
		where the second inequality follows from the $S$ scales property.
		This implies that
		\[|z_r| \ge 4|z_s|
		\qquad \text{and} \qquad
		|z_{r}| \geq 4C_0 \cdot (4C_0^2)^{S-(r+1)}\cdot \delta \]
		for all $r, s \in E$ such that $r < s$.
		The first inequality follows from~\eqref{eq:weakconcentration} combined with~\eqref{eq::zr>4zs}.
		The second inequality follows from~\eqref{eq::zr>4zs} combined with $\|v^{(r+1)}\|_2 \ge (4C_0^2)^{S-(r+1)}\cdot \delta$, by the $S$ scales property.
		In particular, from the last inequality it follows that $|z_{r}| > 3b\delta$, whenever $r \in E$ and $r < S-\log(3b)$.
		
		Let $R$ be the set of indices $r \in E$ so that $|z_r| > 3 b \delta$.
		Now we show that $R$ is large with high probability.
		Observe that the indicators $(1_{\{r \in E\}})_{r \in [S]}$ are independent from each other.
		By Claim~\ref{clm:positiveProb}, we have $\Pr{r \in E} \ge C_0^{-1}$, and
		by Chernoff's inequality\footnote{Chernoff's inequality states that $\Pr{X\le \E(X)/2} \le e^{-\E(X)/8}$, where $X$ is a sum of independent Bernoulli random variables.} we have
		\[\Pr{|E|\le \dfrac{S}{2C_0}} \le e^{-S/(8C_0)}.\]
		As $|z_{r}| > 3b\delta$ whenever $r < S-\log(3b)$, we have
		\begin{align}\label{Rislarge}
			\Pr{|R| \ge \dfrac{S}{2C_0} - \log(3b)} \ge 1-e^{-S/(8C_0)}.
		\end{align}
		
		Let $\eps \in \{-1,1\}^S$ be uniformly chosen, independently of $x \in \{0,1\}^n$.
		Observe that the variables $x_iv_i-v_i/2$ and $\eps_r(x_iv_i-v_i/2)$ 
		are uniformly distributed in $\{- v_i/2,v_i/2\}$, for $i \in [n]$ and $r \in [S]$.
		Moreover, $z_r$ has the same distribution as $\eps_r z_r$, for $r \in [S]$.
		Thus, it suffices to bound the probability that
		$|\eps_1z_1+\ldots+\eps_Sz_S-a|<b\delta$.
		To do so, we first reveal $x \in \{0,1\}^n$, and hence $R$. Then, we reveal $\eps_r$ for $r \notin R$.
		Conditioning on these variables, it suffices to bound the probability that
		\begin{align}\label{eqn:Revent}
			\left |\sum \limits_{r \in R} \eps_r z_r -c \right| <b\delta
		\end{align} 
		for all $c\in \mathbb{R}$.
		
		Fix some $c \in \mathbb{R}$. We claim the following.
		
		\begin{claim} \label{claim:boundedprob}
			The probability that~\eqref{eqn:Revent} occurs is bounded by $2^{-|R|}$, that is,
			there is at most one choice for $\eps \in \{-1,+1\}^R$ so that~\eqref{eqn:Revent} holds. 
		\end{claim}
		
		Observe that this claim together with~\eqref{Rislarge} would imply that
		\[ \P_{x \sim \{0, 1\}^n}\big (|\ip{x}{v} - a| < b \delta\big) < \exp \left (- \dfrac{S}{8C_0} \right) + \exp \left (- \dfrac{S}{2C_0} + \log(3b) \right), \]
		which proves the lemma.
		Thus, it suffices to prove Claim~\ref{claim:boundedprob}.
		
		\textit{Proof of Claim~\ref{claim:boundedprob}.} Suppose for contradiction that there exist $\eps \neq \eps'$, both satisfying~\eqref{eqn:Revent}. Let $r_0$ be the minimum index where $\eps_{r_0}\neq \eps'_{r_0}$, so $|\eps_{r_0}-\eps'_{r_0}|=2$. Then,
		\begin{align}\label{ineq:manyscales-1}
			2 b \delta > \Big| \sum_{r \in R} (\eps_r-\eps'_r) z_r  \Big| \ge  2|z_{r_0}| -  \Big | \sum_{r \in R: r > r_0} (\eps_r-\eps'_r) z_r  \Big|. 
		\end{align}
		However, since $|z_r|\ge 4|z_s|$ for every $r,s\in R$ with $r<s$, we have 
		\begin{align}\label{ineq:manyscales-2}
			\Big| \sum_{r \in R: r > r_0} (\eps_r-\eps'_r) z_r \Big| \le 2|z_{r_0}|\cdot \sum_{i \ge 1} 4^{-i} = \frac{2|z_{r_0}|}{3}.
		\end{align}
		By~\eqref{ineq:manyscales-1} and~\eqref{ineq:manyscales-2}, we conclude that 
		\[	2 b \delta > \Big| \sum_{r \in R} (\eps_r-\eps'_r) z_r  \Big| \ge  \frac{4|z_{r_0}|}{3} > 4b\delta,\]
		a contradiction.
		The last inequality follows from the fact $r_0 \in R$.
	\end{proof}
	
	\section{A lemma from convex geometry} \label{sec:bang}
	
	In this section, we introduce Bang's Lemma, which is the main tool of the proof.
	This lemma was obtained by Bang~\cite{bang1951solution} in his proof of the symmetric case of Tarski's plank problem, as observed by Ball~\cite{ball1991plank}.
	Roughly speaking, it states that for every $k \times k$ symmetric matrix $M$ with positive diagonal entries and a vector $\zeta \in \mathbb{R}^k$, there exists a vector $\eps \in \{\pm 1\}^k$ for which $(M\eps)_t$ is far from $\zeta_t$, for all $t \in [k]$.
	That is, $\eps$ is far from the hyperplanes $\ip{M_i}{x}=\zeta_i$, where $M_i$ denotes the $i$-th row of $M$.
	
	\begin{lemma}[Bang~\cite{ball1991plank,bang1951solution}] \label{lem:bang}
		Let $M$ be a $k \times k$ symmetric matrix such that $M_{tt}\ge 0$ for every $t \in [k]$.
		For $\zeta=(\zeta_1, \ldots, \zeta_k) \in \R^k$ and $\theta=(\theta_1, \ldots, \theta_k) \in \R_{\ge 0}^k$, there exists $\eps \in \{\pm 1\}^k$ such that
		\begin{align} \label{eq:bang}
			\big| \big(M(\theta \eps)\big)_t -\zeta_t \big| \ge M_{tt} \theta_t 
		\end{align}
		for every $t \in [k]$, where $\theta \eps:=( \theta_1 \eps_1,\ldots, \theta_k \eps_k)$.
	\end{lemma}
	
	\begin{proof}
		Let $\eps \in \{\pm 1 \}^k$ be a vector that maximizes the expression
		$$ \ip{M (\theta \eps)}{\theta \eps} - 2 \ip{\theta \eps}{\zeta} = \sum_{i,j=1}^k M_{ij}\theta_i \theta_j \eps_i \eps_j - 2 \sum_{i=1}^k \theta_i \eps_i \zeta_i.  $$
		Let $t \in [k]$ and $\eps^{(t)} \in \{\pm 1\}^k$ be the vector which differs from $\eps$ only in the $t$-th coordinate. By the maximality of $\eps$, we have that
		$$ (\eps_t - \eps^{(t)}_t) \left[ \sum_{i \neq t} M_{it} \theta_i \theta_t \eps_i + \sum_{i \neq t} M_{ti} \theta_i \theta_t \eps_i- 2 \theta_t \zeta_t \right] \ge 0.$$
		This is equivalent to
		\begin{equation} \label{eq:lem12}
			4\theta_t \eps_t \left[ \sum_{i \neq t} M_{it} \theta_i\eps_i - \zeta_t \right] =
			4\theta_t  \eps_t \left[ \sum_{i=1}^k M_{it} \theta_i \eps_i - \zeta_t - M_{tt} \theta_t \eps_t \right] \ge 0 .
		\end{equation}
		Simplifying the left-hand side of \eqref{eq:bang}, we conclude that
		$$ \big| \big(M(\theta \eps) \big)_t -\zeta_t \big| = \left| \sum \limits_{i=1}^k M_{it} \theta_i \eps_i - \zeta_t  \right | \ge \eps_t \left( \sum_{i=1}^k M_{it} \theta_i \eps_i - \zeta_t \right)
		\ \stackrel{\mathclap{\eqref{eq:lem12}}}{\ge} \ 
		M_{tt} \theta_t \eps_t^2 = M_{tt}\theta_t ,$$	
		for every $t \in [k]$.
	\end{proof}
	
	Lemma~\ref{lem:bang} above tells us that we can find a vector from $\{\pm 1\}^k$ that is far from each of the hyperplanes $\left\{ x \in \R^k : \sum_{j=1}^k M_{ij}\theta_j x_j = \zeta_i \right\}$ for $i \in [k]$. We will apply Lemma~\ref{lem:bang} to find a vector $y \in [0,1]^n$ that is far away from each of the hyperplanes $\ip{v_i}{x}=\mu_i$, for $i \in [k]$. Note that this vector $y$ is not necessarily a vertex of the cube $\{0,1\}^n$. We define a probability distribution over the $n$-cube $\{0,1\}^n$ based on the vector $y$ which will make it unlikely for a randomly sampled vector to lie in any of the hyperplanes. We use the probabilistic method to conclude, similar to Lemma~\ref{anti-concentration:sparse-vectors}, that a collection of vectors with a certain structure cannot cover many vertices. This intuition is formalized in Proposition~\ref{prop:bang} below. For its proof, we make use of the following version of Hoeffding's inequality (Lemma~\ref{lem:hoeffding}). We highlight that this stronger inequality allows us to improve the logarithmic term on the number of hyperplanes in an essential cover. The slightly weaker lower bound $ \Omega\left(\frac{n^{5/9}}{(\log n)^{2/3}}\right)$ could be obtained by applying Bernstein's inequality~\cite{bernstein1924modification} instead. 
	
	\begin{lemma}[Hoeffding's inequality~\cite{hoeffding1963probability}] \label{lem:hoeffding}
		Let $z_1, \ldots, z_\ell$ be independent zero mean random variables such that, for every $j \in [\ell]$, $a_j \le z_j \le b_j$ almost surely. Then, for every $t>0$ we have 
		$$ \P \left(\sum_j z_j \ge t\right) \le \exp \left( - \frac{2t^2}{\sum_j (b_j-a_j)^2} \right). $$
	\end{lemma}
	
	Let $\mu \in \R^{\ell}$ and $V$ be an $\ell \times m$ matrix whose rows have $\ell_2$-norm 1.
	Roughly speaking, the proposition says that if the $\ell_2$-norm and the support of each column of $V$ are small, then the rows of $Vx = \mu$ do not cover the entire $m$-cube $\{0,1\}^m$.
	The idea of the proof of Theorem~\ref{thm:main} is to apply Proposition~\ref{prop:bang}.
	For every essential matrix, we find a large submatrix whose norm and support of each column are small.
	We then apply Proposition~\ref{prop:bang} to such submatrix and show that there must be a vertex of the $n$-cube which is not covered.
	Below, $v_{* j}$ denotes the vector corresponding to the $j$-th column of the matrix $V$.

	\begin{prop} \label{prop:bang}
		Let $V$ be an $\ell \times m$ matrix such that $\|v_i\|_2 = 1$ for all $i \in [\ell]$, where $v_i$ is the $i$-th row of $V$. 
		Suppose that $\alpha = \max\limits_{j \in [m]} |\supp(v_{*j})|$ and 
		$\beta = \max\limits_{j \in [m]} \| v_{*j}\|_2^2$ satisfy
		\[2\alpha \beta \log (4\ell)\le 1.\]
		Then, for every $\mu \in \R^{\ell}$, the hyperplanes given by the rows of the system $Vx = \mu$ do not cover the entire $m$-cube $\{0,1\}^m$.
	\end{prop}
	
	\begin{proof}
		Let $\theta, \zeta_1,\ldots,\zeta_{\ell} \in \R$ be given by
		$\theta \coloneqq (2\log(4\ell))^{1/2}$ and 
		$\zeta := 2\mu-V\cdot\bar{1}$, where $\bar{1}$ denotes the vector with all coordinates equal to 1. 
		By Bang's Lemma (c.f.~Lemma~\ref{lem:bang}) for $M := VV^T$ and $\theta_i = \theta$ for $i\in [\ell]$, there exists a vector $\eps \in \{\pm 1\}^{\ell}$ such that 
		\begin{align}\label{eq:bangforVVT}
			| (VV^T\theta\eps)_i - \zeta_i | \ge \theta
		\end{align}
		for every $i \in [\ell]$.
		In particular, the vector $y' := \theta V^T \eps$ satisfies $| \ip{v_i}{y'} - \zeta_i| \ge \theta$ for every $i \in [\ell]$. 
		By the Cauchy--Schwartz inequality, we have
		$$\| y' \|_{\infty} = \| \theta V^T \eps \|_{\infty} = \max_{j \in [n]} \left| \theta \sum_{i=1}^{\ell} v_{ij} \eps_i \right| \le \theta\sqrt{\alpha \beta}\le 1.$$
		
		We would like to find a vector of the $m$-cube (i.e. in $\{0,1\}^m$) that satisfies no row from the system of equations $Vx = \mu$.
		First, we use $y'$ to find a vector $y \in [0,1]^m$ which is far from the hyperplanes $\langle v_i, x \rangle = \mu_i$, for $i \in [\ell]$.
		For this, we set $y:=\frac{y'+\bar{1}}{2}$. 
		As $\|y'\|_{\infty}\le 1$, note that $y \in [0,1]^m$.
		By~\eqref{eq:bangforVVT}, we have
		\begin{align}\label{eq:biggerthantheta}
			| (V(2y-\bar{1}))_i - \zeta_i | = | 2(Vy)_i - 2\mu_i | \ge \theta.
		\end{align}

		Define the random vector $w:= y+\delta$, where $\delta \in \mathbb{R}^m$ is a random vector with independent entries and distribution given by
		\begin{align*} 
			\Pr{\delta_i =1-y_i} = y_i \qquad \text{and} \qquad \Pr{\delta_i=-y_i} = 1-y_i.
		\end{align*} 
		Observe that the random variables $\delta_i$ are chosen so that we have $w \in \{0,1\}^m$ and $\Ex(\delta)=0$.
		To finish the proof of Proposition~\ref{prop:bang}, it suffices to show that $\sum_{i=1}^{\ell}\P(\ip{w}{v_i}=\mu_i) < 1$.
		To bound each of the probabilities $\P(\ip{w}{v_i}=\mu_i)$, we first note that
		\begin{align}\label{eq:reducingtohoeffding}
			\P(\ip{w}{v_i}=\mu_i) \le \P\big(| \ip{\delta}{v_i} | = | \mu_i - \ip{y}{v_i} |\big) \le \P\left( |\ip{\delta}{v_i}| \ge \frac{\theta}{2} \right),
		\end{align}
		where we used~\eqref{eq:biggerthantheta} for the last inequality.
		
		The variables $(\delta_j v_{ij})_{j\in [m]}$ are independent, bounded by $-v_{ij} \le \delta_j v_{ij} \le v_{ij}$ and centered (i.e.~$\E[\delta_j v_{ij}]=0$) for every $j \in [\ell]$.
		By~\eqref{eq:reducingtohoeffding} and Hoeffding's inequality~(c.f.~Lemma~\ref{lem:hoeffding}), we have 
		$$ \P(\ip{w}{v_i} = \mu_i) \le \P\left( |\ip{\delta}{v_i}| \ge \frac{\theta}{2} \right) 
		\le 2 \cdot \exp \left( -\dfrac{2(\theta/2)^2}{\sum_{j =1}^{m} v_{ij}^2} \right) 
		= 2\exp \left( \frac{-\theta^2}{2} \right). $$ 
		As $\theta = (2\log(4\ell))^{1/2}$, we obtain
		\[\sum \limits_{i=1}^{\ell}\P\big(\ip{w}{v_i}=\mu_i\big)\le 
		2 \ell \exp \left( \frac{-\theta^2}{2} \right) \le \frac{1}{2}. \qedhere \]
	\end{proof}
	
	Proposition~\ref{prop:bang} above is the main difference in the proof of Theorem~\ref{thm:main} compared to the proof from~\cite{yehuda2021lower}. Roughly speaking, the better dependence on parameters of the condition $2\alpha \beta \log (4\ell) \le 1$ allows us to substantially improve the lower bound on the number of hyperplanes in any essential cover of the cube.
	
	\section{The structure of essential covers} \label{sec:structure}
	
	In this section, we prove structural statements for essential matrices.
	The goal is to decompose an essential matrix in such a way that we find a suitable submatrix where we can apply Proposition~\ref{prop:bang}.
	The decomposition will be done in two steps.
	In the first decomposition, we divide the essential matrix into four blocks, see Subsection~\ref{firstdecomposition}.
	In one of the blocks, we will have small column norm and row norm either $0$ or $1$.
	As every row needs to have norm 1 to apply Proposition~\ref{prop:bang}, we decompose this block even further to deal with the rows of norm $0$.
	This is the second decomposition (see Subsection~\ref{seconddecomposition}), which is the main result of this section.
	
	The proofs of the decomposition lemmas are done algorithmically and have the same flavor as the proof for the lower bound of $n^{1/3}$ given by Linial and Radhakrishnan~\cite{LINIAL2005331}.
	For the reader to be familiarized with the high-level techniques used in the decompositions, we first state and prove this result.
	
	\begin{lemma}[Linial--Radhakrishnan~\cite{LINIAL2005331}.] \label{lem:n1/3}
		An essential cover of the $n$-cube has at least $n^{1/3}$ hyperplanes.
	\end{lemma}
	
	\begin{proof}
		Let $V$ be a $k \times n$ essential matrix.
		As usual, denote its rows by $v_1, \ldots, v_k$.
		For $\ell > 0$ (to be chosen later to be $n^{2/3}$), we run the following algorithm.
		
		\vspace{1.5mm} \noindent \textbf{Initiate} $L_1 \gets \emptyset$, $L_2 \gets [k]$, $M_1 \gets [n]$ and $M_2 \gets \emptyset$.
		
		\noindent \textbf{While} there is $i \in L_2$ such that $|\supp(v_i)\setminus M_2| < \ell$,
		
		$L_1 \gets L_1 \cup \{i\}$ and $L_2 \gets L_2 - \{i\}$;
		
		$M_2\gets M_2\cup \supp(v_i)$ and $M_1 \gets M_1 -  \supp(v_i)$;
		
		\noindent \textbf{Output:} $L_1$, $L_2$, $M_1$ and $M_2$.
		
		\noindent \textbf{end.}
		
		The algorithm outputs a partition of the rows 
		$L_1\cup L_2 = [k]$ and a partition of the columns $M_1\cup M_2 = [n]$ with the following properties.
		If $L_1, M_1 \neq \emptyset$, then the submatrix $V(L_1 \times M_1)$ is identically $0$. If $L_2 \neq \emptyset$, 
		then $M_1 \neq \emptyset$ and each row in $V(L_2 \times M_1)$ has support of size at least $\ell$.
		Moreover, by relabeling the rows in $L_1$ and the columns in $M_2$ if necessary, the vectors in $L_1$ have the following property.
		If $L_1\neq \emptyset$, then we have $M_2 \neq \emptyset$,
		$\supp(v_i) \se M_2$ and 
		$|\supp(v_{i})\setminus \bigcup_{j <i}\supp(v_{j})|< \ell$ for all $i \le |L_1|$.
		Below, Figure \ref{fig:n1/3} represents a partition given by the algorithm.
		
		\begin{figure}[ht!]
			\centering
			\includegraphics[width=12cm]{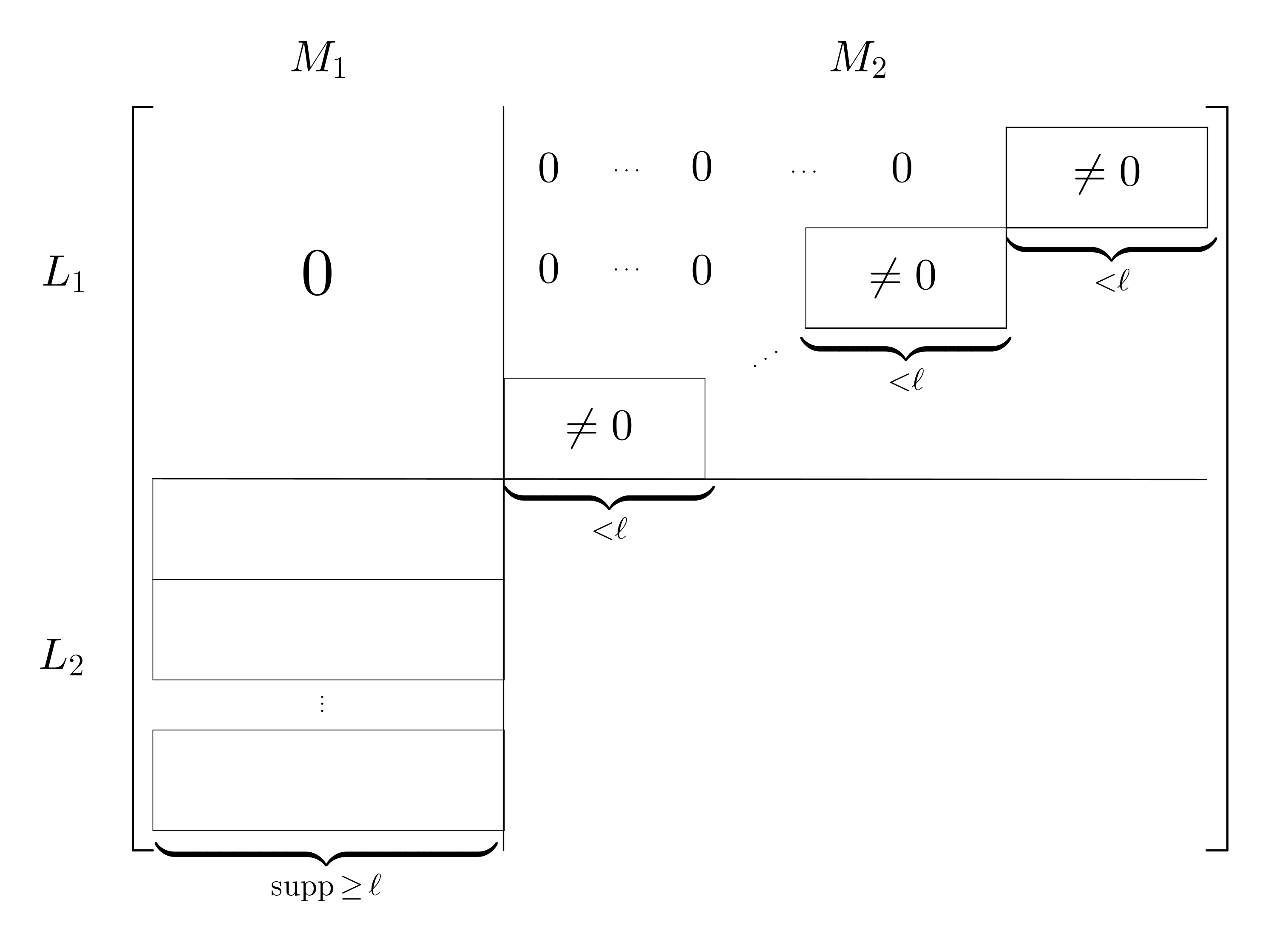}
			\captionsetup{justification=centering,margin=2cm}
			\caption{The decomposition of $V$ in the proof of Lemma \ref{lem:n1/3}.}
			\label{fig:n1/3}
		\end{figure}
		
		We are now ready to bound the size of the essential cover.
		If $L_1 = [k]$, then we have $M_2=[n]$. As $\bigcup_i \supp(v_i) = [n]$, this implies that $k \cdot \ell \ge n$.
		If $L_1 \neq [k]$, then $L_2\neq \emptyset$ and hence there exists $x\in \{0,1\}^n$ not covered by the hyperplanes in $L_1$. Let $P$ be the $|M_1|$-dimensional cube agreeing with $x$ on $M_2$. Fix the $M_2$-coordinates of $x$. Now, the hyperplanes corresponding to the rows in $L_2$ form a covering of $P$. By Lemma \ref{anti-concentration:sparse-vectors}, each hyperplane in $L_2$ covers at most $2^{|M_1|}/\sqrt{\ell}$ vertices of $P$, and hence $L_2$ needs to have size at least $\sqrt{\ell}$ for the $|M_1|$-cube to be covered.
		That is, we have $k \ge |L_2|\ge \sqrt{\ell}.$
		We conclude that
		$k \ge \min\left\{ \frac{n}{\ell} , \sqrt{\ell} \right\}$. Choosing $\ell = n^{2/3}$, we obtain $k \ge n^{1/3}$.
	\end{proof}
	
	Note that the main idea of the proof is to explore the facts that hyperplanes with big support cannot cover many vertices, and that only few hyperplanes can have relatively small support.
	We shall make use of these properties in the first and second decompositions (c.f.~Lemmas~\ref{lemma:decomposing} and~\ref{lemma:seconddecomposition}).
	The next lemma is a key ingredient for the proof of such decompositions.
	It bounds the support of every row in an essential matrix.
	
	\begin{lemma}[Linial--Radhakrishnan~\cite{LINIAL2005331}]\label{lemma:boundsupp} \label{lemma:nati} 
		Let $V$ be a $k \times n$ essential matrix. Then, we have
		$|\supp(v_i)| \le 2 k$ for all $i \in [k]$.
	\end{lemma}
	
	Using Lemma~\ref{lemma:boundsupp}, we can easily show that every essential cover of the $n$-cube has $\Omega(\sqrt{n})$ hyperplanes.
	Indeed, by Lemma~\ref{lemma:boundsupp}, every $k \times n$ essential matrix $V$ has at most $2k^2$ non-zero entries. 
	On the other hand, every column of $V$ has a non-zero entry, and hence we have at least $n$ non-zero entries in total.
	This implies that $2k^2 \ge n$. 
	In~\cite{LINIAL2005331}, Linial and Radhakrishnan were able to obtain a slightly stronger lower bound (by a multiplicative constant) by showing that in every column we have at least two non-zero entries.
	
	\subsection{First decomposition}\label{firstdecomposition}
	
	Let $v_1,\ldots,v_k \in \R^n$ be vectors with $\ell_2$-norm 1 given by an essential cover of the $n$-cube.
	We would like to apply Proposition~\ref{prop:bang} to show that $k$ cannot be small, as otherwise, there will be a vertex not covered by the hyperplanes.
	To apply this proposition, we need to bound the support and the $\ell_2$-norm of each column of $V$.
	However, we might not have good bounds for these quantities: the support of each column can be as large as $k$ and the norm can be as large as $\Omega(\sqrt{k})$.
	Luckily, we expect a random large submatrix of $V$ to be much more well-behaved.
	
	In the first decomposition (c.f.~Lemma~\ref{lemma:decomposing}), we rescale $V$ and find a large submatrix $V[L_1 \times M_1]$ with the following properties.
	The rows of $V[L_1 \times M_1]$ have $\ell_2$-norm either $0$ or $1$; the columns have support roughly bounded by $2k^2/n$; 
	the $\ell_2$-norm of the column vectors is bounded by $W^{-1/2}$, where $W=W(n)>0$ is an arbitrary function;
	and the size of the set $M_2:= M_1^c$ is roughly bounded by $kW$.
	In the proof of Theorem~\ref{thm:main}, the function $W = (n \log n)^{1/9}$ is chosen so that, even if we apply the first decomposition multiple times inside the matrix,
	the union of the sets $M_2$ will have size at most $n/8$.
	This will help us to find a large submatrix whose rows have $\ell_2$-norm $1$ and whose columns have small norm.
	
	Ideally, to apply Proposition~\ref{prop:bang}, we wish that all the rows in $V[L_1 \times M_1]$ had norm $1$.
	However, this cannot be guaranteed in this step, as the first decomposition (c.f.~Lemma~\ref{lemma:decomposing}) holds for every matrix, not necessarily coming from an essential cover.
	In Lemma~\ref{lemma:seconddecomposition}, which is the core of the proof of Theorem~\ref{thm:main} together with Proposition~\ref{prop:bang}, we refine this decomposition to control the rows with many zeros.
	Once we find a submatrix $V[L_1 \times M_1]$ whose rows have norm either 0 or 1 and whose columns have small support, 
	we still need to deal with the rest of the matrix $V$.
	However, the decomposition is done in such a way that if $i \notin L_1$, then the coordinates of $v_i$ have approximately exponential decay.
	
	For a vector $v$ and a set $A$, denote by $\restr{v}{A}$ the subvector of $v$ restricted to the coordinates in $A$.
	We show that there exists $a \in \{0,1\}^{n-|M_1|}$ such that 
	\[\langle \restr{v_i}{M_1}, u \rangle \neq \mu_i - \langle \restr{v_i}{M_1^c}, a \rangle\]
	for every $i\notin L_1$ and $u \in \{0,1\}^{M_1}$.
	Thus, the strategy is reduced to restrict ourselves to the subcube $\{0,1\}^{M_1} \times a$ and show that $V[L_1 \times M_1]$ is not a covering system for this subcube.

	Let $V \in \R^{k \times n}$ be a matrix with rows $v_1, \dots, v_k$ and $V' \in \R^{k \times n}$ be a matrix with rows $v'_1, \dots, v'_k$. We say that $V' \in \R^{k \times n}$ is a \emph{rescaling} of a matrix $V \in \R^{k \times n}$ if there are real numbers $(\phi_i)_{i = 1}^{k}$ such that $v'_i = \phi_i v_i$ for $i \in [k]$. 
	The first decomposition lemma, obtaining a matrix decomposition as in Figure~\ref{fig:1st_decomp}, is as follows.
	
	\begin{lemma}
		\label{lemma:decomposing}
		There exists $C_3>0$ such that the following holds. For every $V \in \R^{k \times n}$, $S\in \N$ and $W>0$, there are partitions $[k] = L_1 \cup L_2$ and $[n] = M_1 \cup M_2$, with $|M_2| \le C_3 kSW$, such that for some rescaling $V'$ of $V$ we have the following.
		\begin{enumerate}[\rm (i)]
			\item \label{itm:oneLem}
			Every row in $V'[L_1 \times M_1]$ has $\ell_2$-norm either $0$ or $1$;
			
			\item Every column in $V'[L_1 \times M_1]$ has $\ell_2$-norm strictly less than $W^{-1/2}$;
			
			\item Every row $i \in L_2$ of $V'$ has $S$ scales, and the position of its smallest scale contains the $M_1$ columns. 
		\end{enumerate}
	\end{lemma}
	
	\begin{figure}[!ht]
		\centering
		\includegraphics[width=12cm]{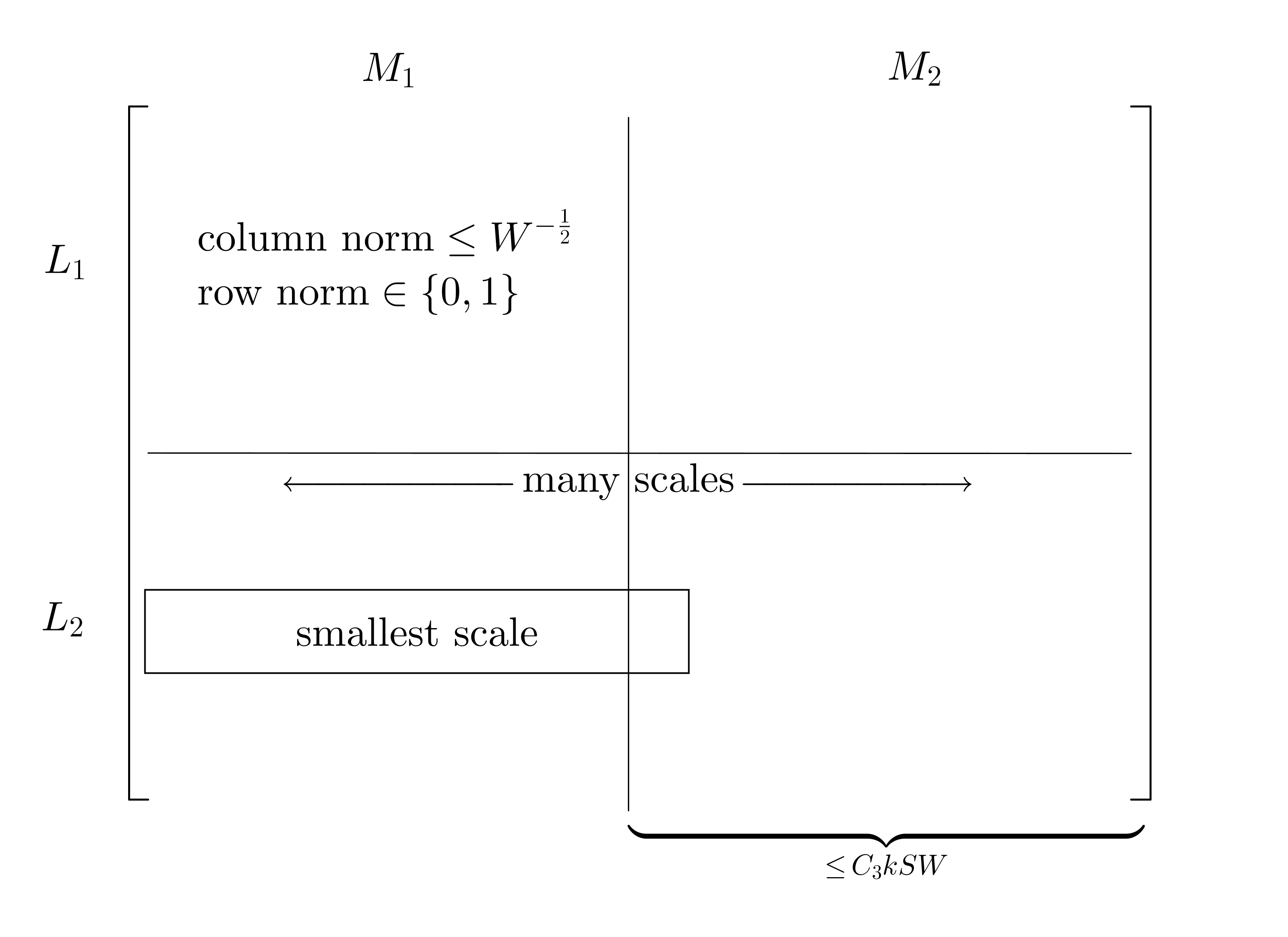}
		\captionsetup{justification=centering,margin=2cm}
		\caption{The decomposition of a rescaling of $V$ as in Lemma~\ref{lemma:decomposing}.}
		\label{fig:1st_decomp}
	\end{figure}
	
	\begin{proof}
		Without loss of generality, assume that each column of $V$ has $\ell_2$-norm equal to $1$.
		The strategy to obtain the desired partitions $[k] = L_1 \cup L_2$ and $[n] = M_1 \cup M_2$ is the following algorithm. Start with $L_1=[k]$ and $M_1=[n]$ and move columns from $M_1$ to $M_2$ until every column in $M_1$ has small norm. While moving those columns, if a row has a certain fraction of its norm moved to $M_2$, then this creates a scale. If we create $S$ scales for a row, then we move that row to $L_2$. 
		
		Formally, the algorithm is as follows. Let $\tau > 0$ be so that $\frac{1-\tau}{\tau}=C_1^2$, where $C_1$ is the constant from Definition~\ref{definition:many_scales}.

		\vspace{1.5mm} \noindent \textbf{Initiate} $L_1 \gets [k]$, $L_2 \gets \emptyset$, $M_1 \gets [n]$ and $M_2 \gets \emptyset$.
		\vspace*{3mm}
		
		\noindent \textbf{While} there exists a column $j \in M_1$ with $\sum_{i \in L_1} {v_{ij}}^2 \geq \tau W^{-1}$, do:
		\vspace*{2mm}
		
		1. Set $M_1 \gets M_1\setminus \{j\}$ and $M_2 \gets M_2\cup \{j\}$. If there is more than one such column $j$, 
		
		then we choose only one of them arbitrarily.
		
		\vspace*{2mm}

		2. For $i \in L_1$ such that $\| \restr{v_i}{M_1} \|_2^2 \in (0,\tau]$,
		multiply the row $i$ by $\| \restr{v_i}{M_1} \|_2^{-2}$.
		That is, set
		\[v_i \gets v_i \cdot \| \restr{v_i}{M_1} \|_2^{-2}.\]
		
		\vspace*{2mm}
		
		3. If Step 2 is executed $S$ times for a row $i$, then set $L_1 \gets L_1\setminus \{i\}$ and $L_2 \gets L_2\cup \{i\}$.
		
		$M_2\gets M_2\cup \supp(v_i)$ and $M_1 \gets M_1 \setminus \supp(v_i)$;
		
		\vspace*{3mm}
		
		\noindent \textbf{For} $i \in L_1$ such that $\| \restr{v_i}{M_1} \|_2 \neq 0$, multiply the row $i$ by $\| \restr{v_i}{M_1} \|_2^{-2}$.
		That is, set
		\[v_i \gets v_i \cdot \| \restr{v_i}{M_1} \|_2^{-1}.\]
		
		\noindent \textbf{Output:} $V$, $L_1$, $L_2$, $M_1$ and $M_2$.
		
		\noindent \textbf{end.}
		
		We now make some simple observations about the algorithm.
		At Step 3 of the \emph{While} loop, if $\| \restr{v_i}{M_1} \|_2^2 \le \tau$ for some $i \in L_1$, then 
		multiplying $v_i$ by $\| \restr{v_i}{M_1} \|_2^{-2}$ renormalizes the vector $v_i$ so that its $\ell_2$-norm restricted to the set $M_1$ is 1. 
		For each $i \in [k]$, denote by $v_i(t)$ and $M_1(t)$ the vector $v_i$ and the set $M_1$ after the $t$-th iteration of the \emph{While} loop, respectively.
		Let $t_1 < t_2$ be two steps corresponding to two consecutive renormalizations of a row $v_i$.
		Before executing Step 3 in iteration $t_2$, we have the following mass distribution in $v_i(t_2-1)$.
		\begin{align*}
			\big \| \restr{v_i(t_2-1)}{M_1(t_1)\setminus M_1(t_2)} \big\|_2^2 \ge 1-\tau
			\qquad \text{and} \qquad
			\big \| \restr{v_i(t_2-1)}{M_1(t_2)} \big\|_2^2 \le \tau.
		\end{align*}
		This implies that
		\begin{align*}
			\big \| \restr{v_i(t_2-1)}{M_1(t_1)\setminus M_1(t_2)} \big\|_2^2 \ge \dfrac{1-\tau}{\tau} \cdot \big \| \restr{v_i(t_2-1)}{M_1(t_2)} \big\|_2^2.
		\end{align*}
		Recall $\frac{1 - \tau}{\tau} = C^2_1$.
		Then, each renormalization of $v_i$ corresponds to a new scale. 
		As we move a vector to $L_2$ if we renormalize it $S$ times, every row in $L_2$ has $S$ scales. Moreover, by construction, the position of the smallest scale contains the $M_1$ columns.
		
		Let $T$ be the last step of the \emph{While} loop and $V'$ be the rescaled matrix output by the algorithm. 
		The renormalization of the rows after the end of the \emph{While} loop
		guarantees that Property~(i) holds.
		Before this renormalization is executed, observe that every column $j \in M_1(T)$ satisfies 
		\[ \sum_{i \in L_1} {v_{ij}}(T)^2 < \tau W^{-1},\]
		and every row $i\in L_1$ has either $\ell_2$-norm zero or 
		\[\sum_{j \in M_1} {v_{ij}}(T)^2 > \tau.\]
		This implies that, after the final renormalization, for each $j \in M_1$
		the sum $\sum_{i \in L_1} {v_{ij}}(T)^2 $ can increase by a factor of at most $1/\tau$, and hence
		\[\sum_{i \in L_1} {v'}^{2}_{ij} < \dfrac{1}{\tau}\cdot \tau W^{-1} = W^{-1}.\]
		This shows that Property~(ii) holds.
		
		It remains to prove that the algorithm ends with $|M_2| \le C_3kSW$, for some absolute constant $C_3>0$. Indeed, each row can be renormalized at most $S$ times and the total sum $\sum v_{ij}^2$ moved from $V[L_1 \times M_1]$ to $V[L_1 \times M_2]$ during the algorithm is at most $kS$. On the other hand, every time a column is moved to $M_2$, the sum $\sum_{\substack{i \in L_1, \\j \in M_1}} v_{ij}^2$ loses a mass of size at least $\tau W^{-1}$. Therefore, we have
		\[|M_2| \cdot \tau W^{-1} \le kS,\]
		and hence the number of columns moved to $M_2$ is at most $kSW/\tau$, as desired.
	\end{proof}
	
	\subsection{Second decomposition}\label{seconddecomposition}
	
	To motivate the second decomposition, let us recall the proof strategy.
	Suppose for contradiction that there exists an essential covering system $Vx=\mu$, where $V \in \R^{k \times n}$ and $k = O \left( \frac{n^{5/9}}{(\log n)^{4/9}} \right)$, for a sufficiently small implicit constant.
	We would like to arrive at a contradiction by finding a vector $u \in \{0,1\}^n$ which is covered by none of the hyperplanes from the system $Vx=\mu$. 
	The first step is to apply Lemma~\ref{lemma:decomposing}.
	Let $L_1 \cup L_2 = [k]$ and $M_1 \cup M_2 = [n]$ be the partition of the rows and columns of $V$ given by Lemma~\ref{lemma:decomposing}, respectively.
	Let $Z$ be the set of rows in $L_1$ which are $0$ when restricted to $M_1$.
	Ideally, we would like to show that for every $i \in Z$ we have
	\begin{align}\label{eq:supbig}
		|\supp(\restr{v_i}{M_2})|>4|Z|^2.
	\end{align}
	If this was true, then we could use the Littlewood--Offord lemma (c.f.~Lemma~\ref{anti-concentration:sparse-vectors}) and Lemma~\ref{lem:probBound} to find a vector $a \in \{0,1\}^{M_2}$ such that 
	\[\langle \restr{v_i}{M_1}, u \rangle \neq \mu_i - \langle \restr{v_i}{M_2}, a \rangle\]
	for every $i\in Z \cup L_2$ and $u \in \{0,1\}^{M_1}$.
	Thus, we could restrict ourselves to the subcube $\{0,1\}^{M_1} \times a$ and apply Proposition~\ref{prop:bang} to $V[(L_1\setminus Z) \times M_1]$ to show that there exists a vector which is not covered.
	Unfortunately, we cannot guarantee that~\eqref{eq:supbig} holds in this step. 
	Instead, we decompose the matrix $V$ much further depending on the size of $Z$.
	If $|Z|$ is big, then we ignore the rows in $Z$ and apply Lemma~\ref{lemma:decomposing} to the matrix $V_1 := V[(L_1\setminus Z) \times M_1]$.
	From Lemma~\ref{lemma:decomposing} we obtain partitions $L_1^{1}\cup L_2^{1}$ and $M_1^{1}\cup M_2^{1}$ of the rows and columns of $V_1$, respectively.
	Similarly, we define $Z_1$ to be the set of rows in $L_1^{1}$ which are $0$ when restricted to $M_1^{1}$.
	If $|Z_1|$ is big, then we apply the same procedure to $V_2 := V[(L_1^{1}\setminus Z_1) \times M_1^{1}]$.
	We repeat this process until we arrive at a submatrix $V_i$ where $|Z_i|$ is small.
	We show that this procedure does not last very long, and hence $V_i$ still has $n/2$ columns.
	Once $|Z_i|$ is small, then we have a better chance of showing that~\eqref{eq:supbig} holds in $V_i$.
	Unfortunately, when we arrive at $V_i$ we still cannot guarantee that~\eqref{eq:supbig} holds in $V_i$. 
	However, since $|Z_i|$ is small, we are only able to show that few rows in $V_i$ do not satisfy~\eqref{eq:supbig}.
	By ignoring these rows and modifying the previous algorithm slightly, we obtain a matrix decomposition as in Figure~\ref{fig:3rd_decomp}.
	Formally, we have the following.
	
	\begin{lemma} \label{lemma:seconddecomposition}
		There exists $C_4>0$ such that the following holds. Let $V \in \R^{k \times n}$ be an essential matrix, $W >0$ and $S\in \N$ be such that $C_3kSW \le n/8$, where $C_3>0$ is the constant given by Lemma~\ref{lemma:decomposing}.
		If $k \le C_4 \cdot (S\cdot W)^{-2/5} \cdot n^{3/5}$, then there exist partitions $[k] = K_1 \cup K_2 \cup K_3 \cup K_4$ and $[n] = N_1 \cup N_2 \cup N_3$ with $|N_1| \geq n/2$ such that for some rescaling $V'$ of $V$ we have the following.
		\begin{enumerate}[\rm (i)]
			\item \label{itm:sparseCol}
			Every column in $N_1 \cup N_2$ has support of size at most $16k^2/n$;
			
			\item $V'[K_1 \times (N_1 \cup N_2)] = 0$ and $V'[K_2 \times N_1] = 0$;
			
			\item In $V'[K_2 \times N_2]$, every row has support of size at least $4|K_2|^2$;
			
			\item  \label{small_l2} In $V'[K_3 \times N_1]$, every row has $\ell_2$-norm 1 and every column has $\ell_2$-norm at most $W^{-1/2}$.
			In particular, item~(\ref{itm:sparseCol}) implies that for every $j \in N_1$ we have
			\begin{align} \label{small_linf}
				\sum_{i \in K_3} |v'_{ij}| < (W^{-1}\cdot 16 k^2/n)^{1/2}.
			\end{align}
			\item \label{item:manyscales} In $V'[K_4 \times (N_1 \cup N_2)]$, every row has $S$ scales, the smallest scale is non-zero and its position contains $N_1$.
		\end{enumerate}
	\end{lemma}
	
	\begin{figure}[ht]
		\centering
		\includegraphics[width=12cm]{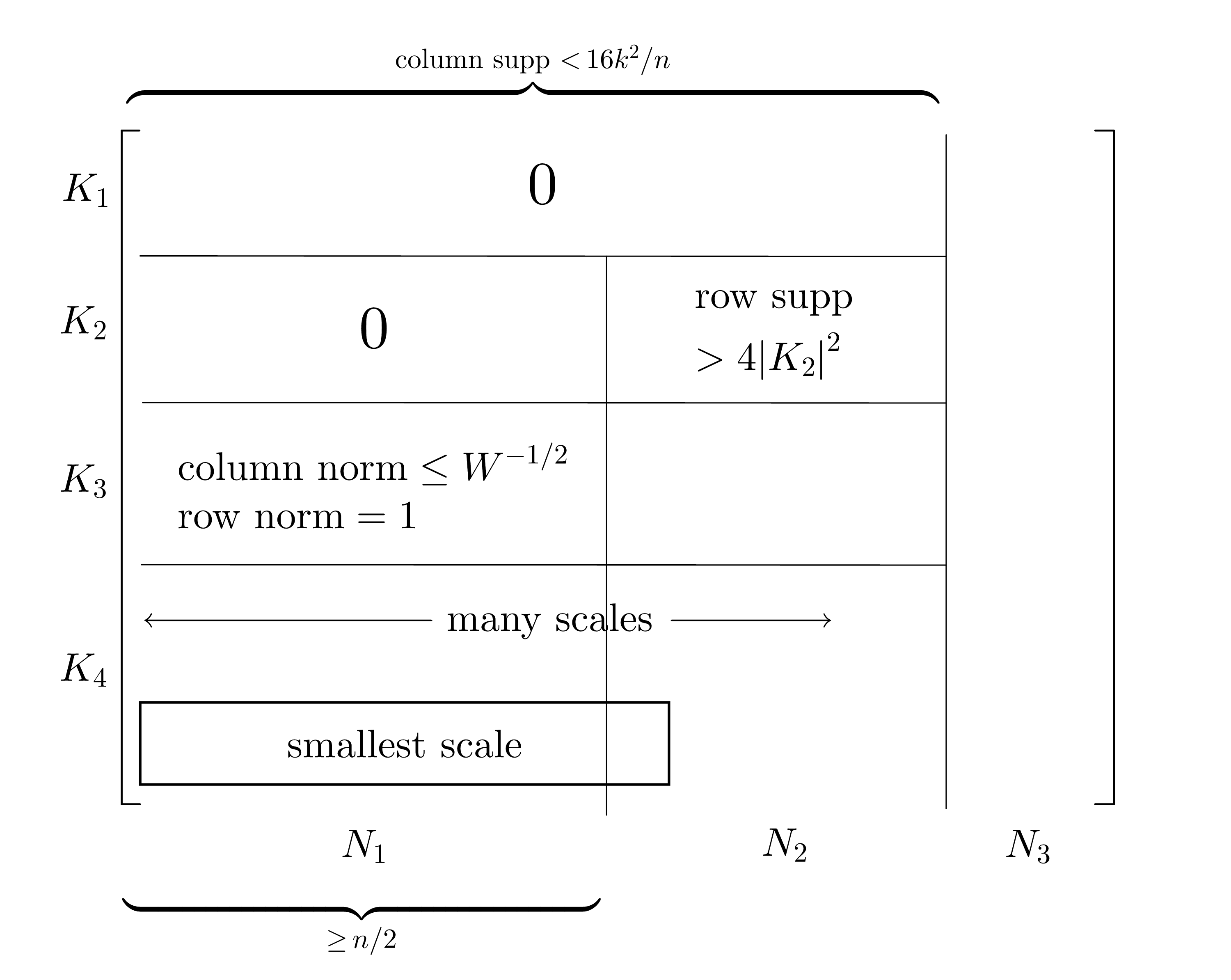}
		\captionsetup{justification=centering,margin=2cm}
		\caption{The decomposition of $V$ in the proof of Lemma~\ref{lemma:seconddecomposition}.}
		\label{fig:3rd_decomp}
	\end{figure}
	
	\begin{proof}
		
		Let $N_3$ be the set of columns with support of size at least $16 k^2 / n$.
		As the number of non-zero entries in $V$ is at most
		$2k^2$ (c.f.~Lemma~\ref{lemma:nati}), we have $|N_3| \le n/8$.
		Let $K_1$ be the set of $i$'s such that $\restr{v_i}{N_3^c} = 0$, where $N_3^c :=[n] \setminus N_3$. 
		Let $\gamma \in (0,1)$ be a parameter to be optimized later (which will be 1/3).
		The partition is obtained by the following algorithm.
		
		\begin{enumerate}
			\item[0.] Set $N_3 \gets \{j\in [n]: |\supp(v_{*j})|\ge 16k^2/n\}$ and $K_1 \gets \{i\in [k]: \restr{v_i}{N_3^c} = 0\}$;
			set $i \gets 0$ and $V_{i} \gets V[K_1^c \times N_3^c]$.
			
			\item[1.] Apply Lemma~\ref{lemma:decomposing} to $V_{i}$.
			Let $V_i'$ be the rescaling of $V_i$ given by Lemma~\ref{lemma:decomposing}. Let $L_1^{i} \cup L_2^{i}$ and $M_1^{i} \cup M_2^{i}$ be the partitions of the rows and columns of $V'_i$, respectively.
			
			\item[2.] Set $Z_i \gets \{j \in [k]: \restr{v_j'}{M_1^i} = 0\}$.
			
			\item[3.] \label{condition1:large_k} 
			If $|Z_i|>|M_2^{i}|^{\gamma}$, then set $K_1 \gets K_1 \cup Z_i$ and $N_3 \gets N_3 \cup M_2^i$.
			
			\item[4.] If $|Z_i| \le |M_2^{i}|^{\gamma}$, then check if there exists $i^{*} \in Z_i$ such that 
			\begin{align}\label{condition2: sparse_V}
				\Big |\supp\left(\restr{v'_{i^*}}{M_2^i}\right) \Big |\leq 4 |Z_i|^2.
			\end{align}
			If it does, then set $K_1 \gets K_1 \cup \{i^{*}\}$ and $N_3 \gets N_3 \cup \supp(v'_{i^{*}})$; if it does not, then proceed to Step 5.
			
			\item[5.] Set $i \gets i+1$ and $V_{i} = V'_i[K_1^c \times N_3^c]$.
			
			\item[6.] Repeat Steps 1--5 until the condition on Step 4 is no longer satisfied and there is no $i^{*}$ as in Step 4.
			
			\item[7.] Set $K_2 = Z_i$, $K_3 = L_2^i$ and $N_2 = M_2^i$. 
			Set $V'$ to be the rescaling of $V$ obtained by performing all operations above. Finally, finish the algorithm.
		\end{enumerate}
		
		Observe that the matrix structure obtained via the algorithm is as in Figure~\ref{fig:3rd_decomp}, which already implies properties (i)--(iv) of the lemma.
		Moreover, any of the sets $K_i$ might be empty, but $K_3$ and $K_4$ cannot be empty at the same time. Otherwise, we would have a column of zeros, a contradiction.
		We emphasize that having some empty sets in the partition is not a problem as long as it does not contradict the essential cover property of the matrix. 
		
		Now, it only remains to show that $|N_1| \ge n/2$.
		We start by proving an upper bound for $|N_3|$.
		As we notice at the beginning of the proof, the initial set $N_3$ has size at most $n/8$.
		Now, we bound the number of columns added to $N_3$ due to Step 4 of the algorithm.
		Let $J$ be the set of $i$'s for which Step 4 is executed.
		By simplicity, set $|M_2^i| = m_i$. 
		Suppose for contradiction that the number of columns added to $N_3$ due to Step 4 is at least $n/8$. That is, $\sum_{i \in J} m_i \ge n/8$.
		Now, observe that
		\begin{align}\label{eq:beta-1}
			k \ge \sum \limits_{i \in J} |Z_i| \ge \sum \limits_{i \in J} m_i^{\gamma}.
		\end{align}
		By Lemma~\ref{lemma:decomposing}, we have $m_i \leq C_3kSW$ for all $i$.
		We shall use this and the following claim to obtain a lower bound for the last sum in~\eqref{lemma:decomposing}.
		
		\begin{claim}\label{claim:boundsum}
			Let $\gamma \in (0,1]$ and $A, B \in \mathbb{R}_{\ge 0}$.
			If $(m_i)_{i \in [t]} \in [0,A]$ and
			$\sum_{i \in [t]} m_i \geq B$, then
			$\sum_{i \in [t]} m_i^\gamma \geq \tfrac{B}{2A} A^{\gamma}$.
		\end{claim}
		
		\begin{proof}
			Partition $[t]$ into sets $R_1,R_2,\ldots,R_Q$ so that for 
			all $q \in [Q]$ we have $A \leq \sum_{i \in R_q} m_i \leq 2A$
			and $Q \geq \tfrac{B}{2A}$.
			As $x^\gamma + y^\gamma \geq (x+y)^\gamma$ for all $x,y \geq 0$, we obtain
			\[ \sum \limits_{i \in [t]} m_i^{\gamma} \ge \sum \limits_{q \in [Q]} \left ( \sum \limits_{i \in R_q} m_i \right)^{\gamma} \ge \dfrac{B}{2A}\cdot A^{\gamma}.\qedhere\]
		\end{proof}
		By Lemma~\ref{lemma:decomposing} and Claim~\ref{claim:boundsum}, we obtain
		$\sum_{i \in J} m_i^{\gamma} \geq \frac{n}{8 C_3 kSW} \cdot (C_3 kSW)^{\gamma}$, which is a contradiction if $\gamma> 0$ is such that
		\begin{align}\label{beta:1st-condition}
			\frac{n}{8 C_3 kSW} \cdot (C_3 kSW)^{\gamma} \geq k.
		\end{align}
		This is the first condition we must have on $\gamma$ to guarantee that Step 4 does not add more than $n/8$ columns to $N_3$.
		The second condition on $\gamma$ is given by the number of columns added to $N_3$ due to Step 5, which is bounded by
		\[\sum_{i \notin J} 4 |Z_i|^2 \leq \sum_{i \notin J} 4 m_i^{{2\gamma}}
		\leq  4 (C_3 kSW)^{2\gamma} k. \]
		Thus, Step 5 adds at most $n/8$ columns to $N_3$ if 
		\begin{align}\label{beta:2st-condition}
			4 (C_3 kSW)^{2\gamma} k \le n/8.
		\end{align}
		
		The last condition we need is that $C_3kSW \le n/8$. In fact, in the last step $T$, we set $N_2 = M_2^T$ and, by Lemma~\ref{lemma:decomposing}, we have $|M_2^T| \le C_3kSW$. 
		Thus, if $C_3kSW \le n/8$, then we guarantee that at Step 7 we remove at most $n/8$ columns from $N_1$. 
		In summary, it follows that Steps 1, 4, 5 and 6 each remove at most $n/8$ columns from $N_1$.
		It follows, if~\eqref{beta:1st-condition} and~\eqref{beta:2st-condition} are satisfied, that $|N_1| \ge n/2$ and $C_3kSW \le n/8$.
		On the other hand,~\eqref{beta:1st-condition} and~\eqref{beta:2st-condition} together are equivalent to 
		\begin{align*}
			\dfrac{k}{n} \le \min \left \{ \dfrac{(C_3 kSW)^{-2\gamma}}{16}, \, 
			\dfrac{(C_3 kSW)^{\gamma -1}}{8}  \right \}.
		\end{align*}
		Hence, the optimal choice of $\gamma = 1/3$ proves the lemma.
	\end{proof}
	
	\section{Proof of Theorem \ref{thm:main}} \label{sec:proof}
	
	Let $k, n, S \in \N$ and $W>0$ be such that $C_3kSW \le n/8$ and $k \le C_4 \cdot (S W)^{-2/5} \cdot n^{3/5}$,
	where the constants $C_3$ and $C_4$ are given by Lemmas~\ref{lemma:decomposing} and~\ref{lemma:seconddecomposition}, respectively.
	Our goal is to optimize the parameters $k$, $S$ and $W$ and use Proposition~\ref{prop:bang} and Lemma~\ref{lemma:seconddecomposition}. 
	The optimal choices are $S=\lfloor C_5 \log n \rfloor$, where $C_5$ is a large constant, and $W=c (n \log n)^{1/9}$, where $c$ is a small constant, which shows that we cannot have an essential cover with
	$k=O \left( \frac{n^{5/9}}{(\log n)^{4/9}} \right)$ hyperplanes. 
	
	Let $V \in \mathbb{R}^{k \times n}$ and suppose that $Vx=\mu$ is an essential covering system where $k = O \left( \frac{n^{5/9}}{(\log n)^{4/9}} \right)$,
	for a sufficiently small implicit constant.
	Let $[k] = K_1 \cup K_2 \cup K_3 \cup K_4$ and $[n] = N_1 \cup N_2 \cup N_3$ be the partitions of the rows and columns of $V$ given by Lemma~\ref{lemma:seconddecomposition}.
	From now on, we assume that $V$ is rescaled in such a way that properties (i)--(v) of Lemma~\ref{lemma:seconddecomposition} hold.
	
	We arrive at a contradiction by finding a vector $u \in \{0,1\}^n$ not covered by the system $Vx=\mu$.
	The coordinates of $u$ are chosen in phases according to the partition $N_1 \cup N_2 \cup N_3 =[n]$. 
	We first choose the coordinates of $\restr{u}{N_3}$ and $\restr{u}{N_2}$ in such a way that the vector $u$ cannot be covered by any of the hyperplanes in $K_1\cup K_2 \cup K_4$. We then apply Proposition~\ref{prop:bang} to the submatrix $K_3 \times N_1$ to show that the hyperplanes cannot cover the entire cube.   
	
	The first step is to choose the coordinates of $u$ which belong to $N_3$, if $N_3 \neq \emptyset$.
	If $N_3$ is empty, we skip this step.
	As we have an essential cover, property (E\ref{essential-condition:minimal}) implies
	that there is a vertex $x$ not covered by the hyperplanes in $K_1$.
	We then set $\restr{u}{N_3} = \restr{x}{N_3}$.
	Notice that no matter how we choose the coordinates $u_j$, for $j\in N_1 \cup N_2$, the vector $u$ is not covered by the hyperplanes in $K_1$, as $V[K_1\times N_3^c]=0$.
	If $K_1 = \emptyset$, then we choose $\restr{u}{N_3}$ arbitrarily.
	
	The second step is to choose the coordinates of $u$ which belong to $N_2$.
	As before, if $N_2$ is empty, we skip this step.
	The next claim states that if $S=\Omega(\log n)$, then there exists $w \in \{0,1\}^{N_2}$ for which the hyperplanes in $K_2 \cup K_4$ do not contain any of the points in the subcube 
	$\{0,1\}^{N_1}\times w \times \restr{u}{N_3}$.
	
	\begin{claim}\label{claim:wexist}
		Let $S = \floor{C_5 \log n}$, where $C_5>0$ is a sufficiently large constant.
		Let $w \in \{0,1\}^{N_2}$ be chosen uniformly at random.
		Then, with positive probability we have 
		\[\langle w, \restr{v_i}{N_2} \rangle \neq \mu_i - \langle \restr{v_i}{N_3}, \restr{u}{N_3} \rangle - \langle \restr{v_i}{N_1}, x \rangle \]
		for all $i \in K_2 \cup K_4$ and all $x \in \{0,1\}^{N_1}$.
	\end{claim}
	
	\begin{proof}
		For simplicity, denote $\mu_i'=\mu_i - \langle \restr{v_i}{N_3}\rangle $.
		By the Littlewood--Offord Lemma (Lemma~\ref{anti-concentration:sparse-vectors}) and a union bound, we have
		\begin{align}\label{eq:k2}
			\Pr{ \bigcup_{i \in K_2}\big \{\langle \restr{v_i}{N_2}, w \rangle = \mu_i'\big \} }\le |K_2| \cdot \frac{1}{\sqrt{4|K_2|^2}} \le \frac{1}{2}.
		\end{align} 
		
		We now fix $i \in K_4$ and bound the probability that $\langle \restr{v_i}{N_2}, w \rangle = \mu_i' - \langle \restr{v_i}{N_1}, x \rangle$ for some $x \in \{0,1\}^{N_1}$.
		By property~(\ref{item:manyscales}), $\restr{v_i}{N_1 \cup N_2}$ has $S$ scales and $\restr{v_i}{N_1}$ is part of its smallest scale. 
		Let $B$ be the part of the smallest scale of $v_i$ outside the $N_1$ columns.
		We can rewrite the equation $\langle \restr{v_i}{N_2}, w \rangle = \mu_i' - \langle \restr{v_i}{N_1}, x \rangle$ as
		\begin{align}\label{eq:rephrasingeq}
			\langle \restr{w}{N_2\setminus B}, \restr{v_i}{N_2\setminus B} \rangle - \mu_i' = -\langle x, \restr{v_i}{N_1} \rangle -\langle \restr{w}{B}, \restr{v_i}{B} \rangle.
		\end{align}
		Now we bound the absolute value of the right-hand side of~\eqref{eq:rephrasingeq}.
		Let $A=[n]\setminus (N_1 \cup B)$, which are the indices consisting of the largest $S-1$ scales of $v_i$.
		Let $\delta>0$ denote the smallest scale of $\restr{v_i}{A}$.
		By the Cauchy--Schwarz inequality, for all $x' \in \{0,1\}^{N_1\cup B}$ we have
		\begin{equation*}
			|\ip{x'}{\restr{v_i}{N_1\cup B}}| \le \sqrt{n} \cdot \|\restr{v_i}{N_1\cup B}\|_2 \le \sqrt{n} \delta.
		\end{equation*}
		This implies that for any choice of $x \in \{0,1\}^{N_1}$ and $\restr{w}{B}\in \{0,1\}^{B}$ we have
		\begin{align}\label{eq:absvalue}
			\big|\langle x, \restr{v_i}{N_1} \rangle +\langle \restr{w}{B}, \restr{v_i}{B} \rangle\big|\le \delta \sqrt{n}.
		\end{align}
		
		From~\eqref{eq:rephrasingeq} and~\eqref{eq:absvalue}, it follows that if $\langle \restr{v_i}{N_2}, w \rangle = \mu_i' - \langle \restr{v_i}{N_1}, x \rangle$ holds for some $x \in \{0,1\}^{N_1}$, then
		\begin{align*}
			|\langle \restr{w}{N_2\setminus B}, \restr{v_i}{N_2\setminus B} \rangle - \mu_i'| \le \delta \sqrt{n}.
		\end{align*}
		Therefore, it suffices to bound the probability that the later inequality occurs.
		By Lemma~\ref{lem:probBound}, we have that
		\begin{align*}
			\P\left( |\langle \restr{w}{N_2\setminus B}, \restr{v_i}{N_2\setminus B} \rangle - \mu_i'| \le \delta \sqrt{n} \right) 
			\le C_2 \exp \Big(-\frac{S-1}{C_2} + C_2 \log \sqrt{n} \Big).
		\end{align*}
		Finally, by choosing $S = \floor{C_5 \log n}$ when $C_5$ is large enough, it follows that the last expression is $o(n^{-1})$.
		We complete the proof by combining this with a union bound over all rows in $K_4$ and \eqref{eq:k2}.
	\end{proof}
	
	Let $w \in \{0,1\}^{N_2}$ be the vector whose existence is given by Claim~\ref{claim:wexist} and set $\restr{u}{N_2}=w$.
	By the choice of $\restr{u}{N_2}$ and $\restr{u}{N_3}$, the hyperplanes in $K_1 \cup K_2 \cup K_4$ do not contain any of the points in the subcube 
	$\{0,1\}^{N_1} \times \restr{u}{N_2 \cup N_3}$.
	Now, we fix the coordinates of $\restr{u}{N_1}$. 
	By Lemma~\ref{lemma:seconddecomposition}, the matrix $V[K_3\times N_1]$ has column norm bounded by $W^{-1/2}$ and the size of the support of each column is upper bounded by $16k^2/n$.
	By Proposition~\ref{prop:bang}, there exists a point $x \in \{0,1\}^{N_1}$ not covered by $V[K_3\times N_1]$ as long as $W\ge (\log n) k^2/n$.
	Recall that, by Lemma~\ref{lemma:seconddecomposition}, we also need to satisfy the conditions $C_3kSW \le n/8$ and $k \le C_4 \cdot (S W)^{-2/5} \cdot n^{3/5}$,
	and hence the smaller the function $W$, the better the bound on $k$.
	Therefore, by choosing $W = (\log n)k^2/n$, the first and second conditions are satisfied when
	\[ \dfrac{c(\log n)k^2}{n} = W \le \min \left \{\dfrac{n}{k(\log n)}, \dfrac{n^{3/2}}{k^{5/2}(\log n)}\right\}, \]
	for a constant $c>0$ (depending on $C_3$, $C_4$ and $C_5$). 
	This holds if $k = O_{c} (n^{5/9}/(\log n)^{4/9})$, and hence we conclude that there exists a vertex $x \in \{0,1\}^{N_1}$ not covered by $V[K_3\times N_1]$.
	By choosing $\restr{u}{N_1}=x$, it follows that $u$ is not covered by the system $Vx=\mu$, which is a contradiction.
	
	\bibliographystyle{abbrv}
	\bibliography{bibli}

\begin{thebibliography}{10}

\bibitem{ALON199379}
N.~Alon and Z.~Füredi.
\newblock Covering the cube by affine hyperplanes.
\newblock {\em European Journal of Combinatorics}, 14(2):79--83, 1993.

\bibitem{ball1991plank}
K.~Ball.
\newblock The plank problem for symmetric bodies.
\newblock {\em Inventiones {M}athematicae}, 104(1):535--543, 1991.

\bibitem{bang1951solution}
T.~Bang.
\newblock A solution of the ``plank problem".
\newblock {\em Proceedings of the American Mathematical Society},
  2(6):990--993, 1951.

\bibitem{bernstein1924modification}
S.~Bernstein.
\newblock On a modification of {C}hebyshev’s inequality and of the error
  formula of {L}aplace.
\newblock {\em Ann. Sci. Inst. Sav. Ukraine, Sect. Math}, 1(4):38--49, 1924.

\bibitem{Erds1945OnAL}
P.~Erd{\H{o}}s.
\newblock On a lemma of {L}ittlewood and {O}fford.
\newblock {\em Bulletin of the American Mathematical Society}, 51:898--902,
  1945.

\bibitem{hoeffding1963probability}
W.~Hoeffding.
\newblock Probability inequalities for sums of bounded random variables.
\newblock {\em Journal of the American Statistical Association},
  58(301):13--30, 1963.

\bibitem{LINIAL2005331}
N.~Linial and J.~Radhakrishnan.
\newblock Essential covers of the cube by hyperplanes.
\newblock {\em Journal of Combinatorial Theory, Series A}, 109(2):331--338,
  2005.

\bibitem{littlewood1943number}
J.~E. Littlewood and A.~C. Offord.
\newblock On the number of real roots of a random algebraic equation (iii).
\newblock {\em Mathematical collection}, 12(3):277--286, 1943.

\bibitem{SAXTON2013971}
D.~Saxton.
\newblock Essential positive covers of the cube.
\newblock {\em Journal of Combinatorial Theory, Series A}, 120(5):971--975,
  2013.

\bibitem{yehuda2021lower}
G.~Yehuda and A.~Yehudayoff.
\newblock A lower bound for essential covers of the cube.
\newblock arXiv: 2105.13615v1, 2021.

\bibitem{yehuda2021slicing}
G.~Yehuda and A.~Yehudayoff.
\newblock Slicing the hypercube is not easy.
\newblock arXiv: 2102.05536v2, 2021.

\end{thebibliography}
\end{document}